\documentclass[a4paper,twoside]{article}
\usepackage{a4}
\usepackage{amssymb}
\usepackage{amsmath}
\usepackage{upref}
\usepackage{graphicx}
\usepackage[active]{srcltx}
\usepackage[pagebackref,colorlinks,citecolor=blue,linkcolor=blue]{hyperref}
\allowdisplaybreaks[2] 
\newcount\minutes \newcount\hours
\hours=\time
\divide\hours 60
\minutes=\hours
\multiply\minutes -60
\advance\minutes \time
\newcommand{\klockan}{\the\hours:{\ifnum\minutes<10 0\fi}\the\minutes}
\newcommand{\tid}{\today\ \klockan}
\newcommand{\prtid}{\smash{\raise 10mm \hbox{\LaTeX ed \tid}}}
\renewcommand{\prtid}{}
\makeatletter
\def\sectionmark#1{} 
\def\subsectionmark#1{}
\newcommand{\sectnr}{\ifnum \c@secnumdepth >\z@
                 \thesection.\hskip 1em\relax \fi}
\def\@evenhead{\footnotesize\rm\thepage\hfil\leftmark\hfil\llap{\prtid}}
\def\@oddhead{\footnotesize\rm\rlap{\prtid}\hfil\rightmark\hfil\thepage}
\def\tableofcontents{\section*{Contents} 
 \@starttoc{toc}}
\makeatother
\makeatletter
\def\@biblabel#1{#1.}
\makeatother
\makeatletter
\let\Thebibliography=\thebibliography
\renewcommand{\thebibliography}[1]{\def\@mkboth##1##2{}\Thebibliography{#1}
\addcontentsline{toc}{section}{References}
\frenchspacing 
\setlength{\@topsep}{0pt}
\setlength{\itemsep}{0pt}%
\setlength{\parskip}{0pt plus 2pt}%
}
\makeatother
\makeatletter
\def\mdots@{\mathinner.\nonscript\!.%
 \ifx\next,.\else\ifx\next;.\else\ifx\next..\else
 \nonscript\!\mathinner.\fi\fi\fi}
\let\ldots\mdots@
\let\cdots\mdots@
\let\dotso\mdots@
\let\dotsb\mdots@
\let\dotsm\mdots@
\let\dotsc\mdots@
\def\vdots{\vbox{\baselineskip2.8\p@ \lineskiplimit\z@
    \kern6\p@\hbox{.}\hbox{.}\hbox{.}\kern3\p@}}
\def\ddots{\mathinner{\mkern1mu\raise8.6\p@\vbox{\kern7\p@\hbox{.}}%
    \raise5.8\p@\hbox{.}\raise3\p@\hbox{.}\mkern1mu}}
\makeatother
\makeatletter
\let\Enumerate=\enumerate
\renewcommand{\enumerate}{\Enumerate%
\setlength{\@topsep}{0pt}
\setlength{\itemsep}{0pt}%
\setlength{\parskip}{0pt plus 1pt}%
\renewcommand{\theenumi}{\textup{(\alph{enumi})}}%
\renewcommand{\labelenumi}{\theenumi}%
}
\let\endEnumerate=\endenumerate
\renewcommand{\endenumerate}{\endEnumerate\unskip}
\makeatother
\makeatletter
\def\@seccntformat#1{\csname the#1\endcsname.\quad}
\makeatother
\makeatletter
\long\def\@makecaption#1#2{%
  \vskip\abovecaptionskip
  \sbox\@tempboxa{ #1. #2}%
  \ifdim \wd\@tempboxa >\hsize
    #1. #2\par
  \else
    \global \@minipagefalse
    \hb@xt@\hsize{\hfil\box\@tempboxa\hfil}%
  \fi
  \vskip\belowcaptionskip}
\makeatother
\newcommand{\authortitle}[2]{\author{#1}\title{#2}\markboth{#1}{#2}}
\newcommand{\auth}[2]{{#1, #2.}}
\newcommand{\art}[6]{{\sc #1, \rm #2, \it #3\/ \bf #4 \rm (#5), \mbox{#6}.}}
\newcommand{\artprep}[3]{{\sc #1, \rm #2, \rm #3.}}
\newcommand{\artin}[3]{{\sc #1, \rm #2,  in #3.}}
\newcommand{\arttoappear}[3]{{\sc #1, \rm #2, to appear in \it #3}}
\newcommand{\book}[3]{{\sc #1, \it #2, \rm #3.}}
\newcommand{\AND}{{\rm and }}
\RequirePackage{amsthm}
\newtheoremstyle{descriptive}%
  {\topsep}   
  {\topsep}   
  {\rmfamily} 
  {}          
  {\bfseries} 
  {.}         
  { }         
  {}          
\newtheoremstyle{propositional}%
  {\topsep}   
  {\topsep}   
  {\itshape}  
  {}          
  {\bfseries} 
  {.}         
  { }         
  {}          
\theoremstyle{propositional}
\newtheorem{thm}{Theorem}[section]
\newtheorem{prop}[thm]{Proposition}
\newtheorem{lem}[thm]{Lemma}
\newtheorem{theorem}[thm]{Theorem}
\newtheorem{cor}[thm]{Corollary}
\theoremstyle{descriptive}
\newtheorem{deff}[thm]{Definition}
\newtheorem{definition}[thm]{Definition}
\newtheorem{remark}[thm]{Remark}
\def\vint{\mathop{\mathchoice%
          {\setbox0\hbox{$\displaystyle\intop$}\kern 0.22\wd0%
           \vcenter{\hrule width 0.6\wd0}\kern -0.82\wd0}%
          {\setbox0\hbox{$\textstyle\intop$}\kern 0.2\wd0%
           \vcenter{\hrule width 0.6\wd0}\kern -0.8\wd0}%
          {\setbox0\hbox{$\scriptstyle\intop$}\kern 0.2\wd0%
           \vcenter{\hrule width 0.6\wd0}\kern -0.8\wd0}%
          {\setbox0\hbox{$\scriptscriptstyle\intop$}\kern 0.2\wd0%
           \vcenter{\hrule width 0.6\wd0}\kern -0.8\wd0}}%
          \mathopen{}\int}
{\catcode`p =12 \catcode`t =12 \gdef\eeaa#1pt{#1}}      
\def\accentadjtext#1{\setbox0\hbox{$#1$}\kern   
                \expandafter\eeaa\the\fontdimen1\textfont1 \ht0 }
\def\accentadjscript#1{\setbox0\hbox{$#1$}\kern 
                \expandafter\eeaa\the\fontdimen1\scriptfont1 \ht0 }
\def\accentadjscriptscript#1{\setbox0\hbox{$#1$}\kern   
                \expandafter\eeaa\the\fontdimen1\scriptscriptfont1 \ht0 }
\def\accentadjtextback#1{\setbox0\hbox{$#1$}\kern       
                -\expandafter\eeaa\the\fontdimen1\textfont1 \ht0 }
\def\accentadjscriptback#1{\setbox0\hbox{$#1$}\kern     
                -\expandafter\eeaa\the\fontdimen1\scriptfont1 \ht0 }
\def\accentadjscriptscriptback#1{\setbox0\hbox{$#1$}\kern 
                -\expandafter\eeaa\the\fontdimen1\scriptscriptfont1 \ht0 }
\def\itoverline#1{{\mathsurround0pt\mathchoice
        {\rlap{$\accentadjtext{\displaystyle #1}
                \accentadjtext{\vrule height1.593pt}
                \overline{\phantom{\displaystyle #1}
                \accentadjtextback{\displaystyle #1}}$}{#1}}
        {\rlap{$\accentadjtext{\textstyle #1}
                \accentadjtext{\vrule height1.593pt}
                \overline{\phantom{\textstyle #1}
                \accentadjtextback{\textstyle #1}}$}{#1}}
        {\rlap{$\accentadjscript{\scriptstyle #1}
                \accentadjscript{\vrule height1.593pt}
                \overline{\phantom{\scriptstyle #1}
                \accentadjscriptback{\scriptstyle #1}}$}{#1}}
        {\rlap{$\accentadjscriptscript{\scriptscriptstyle #1}
                \accentadjscriptscript{\vrule height1.593pt}
                \overline{\phantom{\scriptscriptstyle #1}
                \accentadjscriptscriptback{\scriptscriptstyle #1}}$}{#1}}}}
\def\itunderline#1{{\mathsurround0pt\mathchoice
        {\rlap{$\underline{\phantom{\displaystyle #1}
                \accentadjtextback{\displaystyle #1}}$}{#1}}
        {\rlap{$\underline{\phantom{\textstyle #1}
                \accentadjtextback{\textstyle #1}}$}{#1}}
        {\rlap{$\underline{\phantom{\scriptstyle #1}
                \accentadjscriptback{\scriptstyle #1}}$}{#1}}
        {\rlap{$\underline{\phantom{\scriptscriptstyle #1}
                \accentadjscriptscriptback{\scriptscriptstyle #1}}$}{#1}}}}
\numberwithin{equation}{section}
\newcommand{\p}{{$p\mspace{1mu}$}}
\newcommand{\R}{{\mathbf R}}
\DeclareMathOperator*{\essliminf}{ess\,lim\,inf}
\DeclareMathOperator{\Div}{div}
\newcommand{\eps}{{\varepsilon}}
\newcommand{\Rn}{\R^n}
\newcommand{\vp}{\varphi}
\renewcommand{\phi}{\varphi}
\newcommand{\la}{\lambda}
\newcommand{\La}{\Lambda}
\newcommand{\vt}{\tilde{v}}
\newcommand{\ut}{\tilde{u}}
\newcommand{\bdry}{\partial}
\newcommand{\bdy}{\bdry}
\newcommand{\bdyp}{\bdy_p}
\newcommand{\Rno}{\mathbf{R}^{n+1}}
\newcommand{\setm}{\setminus}
\newcommand{\uP}{\itoverline{P}}
\newcommand{\lP}{\itunderline{P}}
\newcommand{\de}{\delta}
\newcommand{\UU}{\mathcal{U}}
\newcommand{\LL}{\mathcal{L}}
\newcommand{\Thetaminus}{\Theta_-}
\newcommand{\Thetah}{\widehat{\Theta}}
\newcommand{\Thetatminus}{\widetilde{\Theta}_-}
\newcommand{\ga}{\gamma}
\newcommand{\al}{\alpha}
\newcommand{\alp}{\alpha}
\newcommand{\clV}{\overline{V}}
\newcommand{\lt}{\bigl|{\log |t|}\bigr|}
\newcommand{\ltsub}{|{\log |t|}|}
\renewcommand{\emptyset}{\varnothing}
\newcommand{\xh}{\hat{x}}

\newenvironment{ack}{\medskip{\it Acknowledgement.}}{}

\begin{document}

\authortitle{Anders Bj\"orn, Jana Bj\"orn and Mikko Parviainen}
{The tusk condition and Petrovski\u\i\ criterion 
for the normalized \p-parabolic equation}

\author{
Anders Bj\"orn \\
\it\small Department of Mathematics, Link\"oping University, \\
\it\small SE-581 83 Link\"oping, Sweden\/{\rm ;}
\it \small anders.bjorn@liu.se
\\
\\
Jana Bj\"orn \\
\it\small Department of Mathematics, Link\"oping University, \\
\it\small SE-581 83 Link\"oping, Sweden\/{\rm ;}
\it \small jana.bjorn@liu.se
\\
\\
Mikko Parviainen \\
\it\small Department of Mathematics and Statistics, University of Jyv\"askyl\"a,\\
\it\small P.O. Box 35\/  \textup{(}MaD\/\textup{)}, 
  FI-40014 Jyv\"askyl\"a, Finland\/{\rm ;}
\it\small mikko.j.parviainen@jyu.fi
}

\date{} 

\maketitle

\noindent{\small
{\bf Abstract}.
We study boundary regularity for 
the normalized \p-parabolic 
equation in arbitrary bounded domains.
Effros and Kazdan 
(\emph{Indiana Univ.\ Math.\ J.} {\bf 20} (1970), 683--693)
showed that the so-called
tusk condition guarantees regularity for the heat equation.
We generalize this result 
to the normalized \p-parabolic 
equation, and also obtain H\"older continuity.
The tusk condition is a parabolic version of the exterior cone condition.
We also obtain a sharp Petrovski\u\i\ criterion 
for the regularity of the latest moment of a domain. 
This criterion
implies that the regularity of a boundary point is 
affected if one side of the equation is multiplied by a constant.
}

\bigskip
\noindent
{\small \emph{Key words and phrases}: 
boundary regularity,
exterior ball condition, 
H\"older continuity up to  the boundary, 
normalized \p-parabolic equation,
parabolic comparison principle,
  pasting lemma,
Petrovski\u\i\ criterion, 
strong minimum principle,
tusk condition,
viscosity solution.

\medskip
\noindent
{\small Mathematics Subject Classification (2010):
Primary: 35K61,
Secondary: 35B30, 35B51,  35D40, 
35K92.
}
}

\section{Introduction}

Let $\Theta$ be a bounded open set in a Euclidean space
and for every $f \in C(\bdy \Theta)$
let $u_f$ be the solution of the Dirichlet problem for a given 
partial differential equation.
Then a boundary point $\xi_0 \in \bdy \Theta$ is \emph{regular} if
\[
             \lim_{\Theta \ni \zeta \to \xi_0} u_f(\zeta)=f(\xi_0)
             \quad \text{for all } f \in C(\bdy \Theta),
\]
i.e.\ if the solution to the Dirichlet problem attains
the given boundary data continuously.
In other words, the Dirichlet problem is solvable in the classical sense
if and only if all boundary points are regular,
in which case $\Theta$ is called \emph{regular}.
For solving the Dirichlet problem  in this context we use Perron solutions.

In this paper, we consider boundary regularity for the 
\emph{normalized \p-parabolic equation}
\begin{equation}
\label{eq:normalized-p-parabolic}
u_t= \Delta_p^N u,
\end{equation}
where  $1<p<\infty$ and, at least as long as $\nabla u\neq 0$, 
\begin{align*}
\Delta_p^N u & =\Delta u+(p-2)\Delta_\infty^N u=|\nabla u|^{2-p} \Delta_p u, \\
\Delta_\infty^N u& =|\nabla  u|^{-2}\langle D^2u \nabla u,\nabla u \rangle, \\
\Delta_p u &= \Div (|\nabla u|^{p-2} \nabla u).
\end{align*}
For $p=2$ both the normalized and the non-normalized \p-parabolic
equation 
\begin{equation}
\label{eq:p-parabolic}
u_t= \Delta_p u,
\end{equation}
reduce to the heat equation.

The normalized \p-parabolic equation
\eqref{eq:normalized-p-parabolic} has applications in, e.g.,
image processing, see Does~\cite{does11}, and arises
from tug-of-war games with noise,
see Manfredi--Parviainen--Rossi~\cite{manfredipr10c}.
Compared with \eqref{eq:p-parabolic}, it
has the advantage that solutions remain solutions when multiplied by constants.
On the other hand, it
is in nondivergence form and the
solutions are understood in the viscosity sense. 
Moreover, the normalized
\p-Laplacian $\Delta_p^N u$ is discontinuous at the zeros of the gradient
$\nabla u$. 

Boundary regularity for the normalized \p-parabolic equation 
\eqref{eq:normalized-p-parabolic}
was first studied by 
 Banerjee--Garofalo~\cite{banerjeeg15}, see also their earlier paper
\cite{banerjeeg13}.
More  recently Jin--Silvestre~\cite{jins} established the interior 
 $C^{1,\alp}$-regularity for solutions of \eqref{eq:normalized-p-parabolic}, 
see also   
Imbert--Jin--Silvestre~\cite{imbertjs}, Attouchi--Parviainen~\cite{attouchip}, and  Parviainen--Ruosteenoja~\cite{parviainenr16}  for related regularity results.

The following is our main result.

\begin{thm}
\label{thm-tusk-intro}
\textup{(The tusk condition)}
Assume that
$\Theta$ satisfies the tusk condition at $\xi_0=(0,0)$, i.e.\ 
there are $\xh \in \R^n$ and  $R,T>0$ such that the tusk
\[
\{(x,t)\in\R^{n+1}: -T<t<0 \text{ and } |x-(-t)^{1/2} \xh |^2 < R^2(-t)\} 
\subset \R^{n+1} \setm \Theta.
\]
Then $\xi_0$ is regular.
Moreover, if $f:\bdry\Theta\to\R$ is bounded and H\"older continuous 
at $\xi_0$ then so is the upper Perron solution $\uP f$.
\end{thm}

For the heat equation, Effros--Kazdan~\cite{effrosk71} 
showed that the very same tusk condition implies boundary regularity
(but without H\"older continuity part)
and Lieberman~\cite{lieberman89} generalized this
to linear uniformly parabolic equations.

Our proof is based on the strong minimum principle and 
the fact that the shape of the tusk is invariant under parabolic scaling.
Compared with the proof in \cite{effrosk71} we do not
have their removability Lemma~1 at our disposal, and instead
we use the parabolic comparison principle.
We also need to first deduce the strong minimum 
principle, the parabolic comparison principle
and show that the exterior ball condition guarantees boundary regularity.
The exterior ball condition is much more restrictive
than the tusk condition, but it is needed in our proof.
We also improve on the earlier results (including
the heat equation) by showing the H\"older continuity at
boundary points with an external tusk.

In this paper, we also deduce the following generalization of 
Petrovski\u\i's  criterion, which for the heat equation
was proved in~\cite{petrovskii35}.

\begin{thm} \label{thm-Petr-intro}
\textup{(The Petrovski\u\i\  criterion)}
Let 
\[
   \Theta :=\{(x,t) : |x|^2 <  A|t| \log {\lt}
   \text{ and } -\tfrac13 < t < 0\},
\]
where $A>0$.
Then $\xi_0=(0,0)$ is regular if and only if
$A \le 4(p-1)$.
\end{thm}

As in the case of the heat equation, it follows directly from the 
Petrovski\u\i\  criterion that regularity is different
for the normalized \p-parabolic equation~\eqref{eq:normalized-p-parabolic}
and for its multiplied cousins
\[
    a u_t= \Delta_p^N u, \quad a>0, \ a \ne 1.
\] 
This is in great contrast to the situation for 
the non-normalized \p-parabolic equation~\eqref{eq:p-parabolic}, 
with $p \ne 2$, for which it was shown by
Bj\"orn--Bj\"orn--Gianazza--Par\-vi\-ain\-en~\cite[Theorem~3.6]{BBGP}
that it and all its cousins have the same regular points.

The natural parabolic scaling for 
\eqref{eq:p-parabolic}
takes on a different form than for 
\eqref{eq:normalized-p-parabolic}
and the heat equation (both of which are 
invariant under the same parabolic scaling), 
see Bj\"orn--Bj\"orn--Gianazza~\cite[Section~4]{BBG}.
This is one reason for why the Petrovski\u\i\  criterion for 
\eqref{eq:p-parabolic},
obtained in  \cite[Theorem~1.1]{BBG}, 
looks quite different from Theorem~\ref{thm-Petr-intro} above.
In particular, the constant $A$ in the Petrovski\u\i\ criterion 
for~\eqref{eq:p-parabolic} is unimportant and instead it is the power
of $|x|$ that determines the regularity.
Moreover, it follows from that result
that the tusk condition (corresponding to the natural parabolic
scaling) does not imply regularity for \eqref{eq:p-parabolic}, at least not
for $p>2$.

A key tool for all these boundary regularity results is
the barrier characterization saying that a boundary point is regular if
and only if it has a barrier, 
or a family of barriers in the case of \eqref{eq:p-parabolic}.
For 
\eqref{eq:normalized-p-parabolic},
this   characterization was obtained by 
Banerjee--Garofalo~\cite[Theorem~4.5]{banerjeeg15},
while for the heat equation it is due to 
Bauer~\cite[Theorems~30 and~31]{Bauer62}.
On the contrary, 
regularity for \eqref{eq:p-parabolic} is characterized by a family of 
barriers~\cite[Theorem~3.3]{BBGP}, 
while one (traditional) barrier is not enough  by \cite[Proposition~1.2]{BBG},
at least not for $p<2$.

Thus it seems that the boundary regularity theory for
the normalized \p-parabolic equation 
\eqref{eq:normalized-p-parabolic} is much more
similar to the theory for the heat equation than for
the non-normalized \p-parabolic equation \eqref{eq:p-parabolic}.
However, this is not the complete picture.
The main result in Banerjee--Garofalo~\cite[Theorem~1.1]{banerjeeg15}
says that a lateral point $(x_0,t_0)$ of a
space-time cylinder $G_{t_1,t_2}:= G \times (t_1,t_2) \subset \R^{n+1}$
(i.e.\ $x_0 \in \bdy G$ and $t_1 \le t_0 \le t_2$) is
regular for 
\eqref{eq:normalized-p-parabolic}
if and only if $x_0$ is regular for
\p-harmonic functions in $G$, when $p \ge 2$.
The very same criterion was obtained
for \eqref{eq:p-parabolic}
by Kilpel\"ainen--Lindqvist~\cite{KiLi96}
and Bj\"orn--Bj\"orn--Gianazza--Par\-vi\-ain\-en~\cite[Theorem~3.9]{BBGP},
for all $1<p<\infty$.

The boundary regularity theory for the heat equation has a long 
and colourful history, see e.g.\ Watson~\cite{watson12}. 
Since we have not been able to find a suitable reference
containing all the details mentioned below, we state them here.
Petrovski\u\i's criterion~\cite{petrovskii35} dates back to 1935.
Soon afterwards, 
Tikhonov~\cite[Theorems~1--3]{Tikhonov38} in 1938, showed that the parabolic
boundary of a cylinder $G\times(t_1,t_2)$ is regular for the heat equation 
if and only if $G$ is regular for harmonic functions. 
That a specific point on the lateral boundary of a cylinder is regular if and only if
the corresponding base point is regular for harmonic functions
was shown by Babu\v{s}ka--V\'yborn\'y~\cite{BabVyb} in 1962.
They proved a barrier
characterization in cylinders, and used it to prove their
regularity result.
The same year, Bauer~\cite[Theorems~30 and~31]{Bauer62}
obtained the general barrier characterization for 
the heat equation.

Evans--Gariepy~\cite{evansg82} obtained
the Wiener criterion for the heat equation,
and Fabes--Garofalo--Lanconelli~\cite{FaGaLa}
generalized this to linear uniformly parabolic
equations with $C^1$-Dini coefficients.
A different type of Wiener criterion for the heat equation has also been obtained 
by Landis~\cite{landis69}, \cite{landis91}.
However, contrary to the (linear and nonlinear) elliptic case, where
the Wiener criterion is really useful to deduce (ir)regularity, 
the parabolic Wiener criterion seems to be much harder to use in practice,
and the preferred way  of  deducing (ir)regularity
is using barriers, at least in most situations.

In a very recent preprint, Ubostad~\cite{ubostad} independently 
studies the Perron method and the
Petrovski\u\i's criterion for the multiplied equation
$u_t=p^{-1}\Delta_p^N u$.

\begin{ack}
The first two authors 
were supported by the Swedish Research Council, 
grants 621-2014-3974 and 2016-03424. 
The third author was partly supported by the Academy of Finland project \#260791.
This project started in the fall of 2015, and was in particular
conducted during several visits of the authors to Jyv\"askyl\"a resp.\
Link\"oping  in 2015--16. We thank these institutions for their kind hospitality.
\end{ack}

\section{Preliminaries}

Let  $\Theta\subset \R^{n+1}$ be an open  set  
and  $1<p<\infty$.
Points in $\R^{n+1}$ are written as $\xi=(x,t)$.
We let $\nabla u$ denote the gradient of $u$ in the space directions,
while $D^2 u$ is the matrix of all second derivatives in the space directions.
Also the operators $\Delta_p$ and $\Delta^N_p$ are considered with respect
to the space variable $x$.

Next we recall the definition of a viscosity (super/sub)solution 
to \eqref{eq:normalized-p-parabolic}. 
If the gradient of a test function vanishes, then we need to specify how 
to interpret the equation. To this end, we use the standard way, see 
Giga~\cite{giga06} and Crandall--Ishii--Lions~\cite{crandallil92}, 
of replacing the operator by its lower/upper semicontinuous envelope.   

\begin{definition}  \label{deff-visc-sol}
A function $u:\Theta\rightarrow (-\infty,\infty]$
is a \emph{viscosity supersolution} to  \eqref{eq:normalized-p-parabolic}   
in $\Theta$ if
\begin{enumerate}
\renewcommand{\theenumi}{\textup{(\roman{enumi})}}%
\item $u$ is lower semicontinuous;
\item $u$ is finite in a dense subset of ${\Theta}$;
\item for all $(x_0,t_0)\in\Theta $  and $\varphi\in C^2(\Theta)$,
such that
  $u(x_0, t_0) = \vp(x_0, t_0)$  and $u(x, t) > \varphi(x, t)$ for 
$(x, t) \in \Theta \setm \{(x_0,t_0)\}$,
 we have 
\[
\begin{cases}
\phi_t(x_0,t_0)-\Delta_p^N \varphi(x_0,t_0)\geq 0, &
		\text{if } \nabla \varphi(x_0,t_0)\neq 0,\\
\phi_t(x_0,t_0)-\Delta\phi(x_0,t_0)-(p-2)\lambda_{\text{eig}}
\geq 0,
		&\text{if } \nabla \varphi(x_0,t_0)=0,
\end{cases}
\]
\end{enumerate}
\medskip 
where $\lambda_{\text{eig}}$ denotes the smallest (if $p \ge 2$)
or the largest (if $1<p<2$)
eigenvalue of the matrix $D^2\varphi(x_0,t_0)$.

A function $u$ is a \emph{viscosity subsolution} 
if $-u$ is a viscosity supersolution,
and a \emph{viscosity solution} if it is both a viscosity sub- and
		supersolution.
\end{definition}

We sometimes 
briefly refer to the conditions above for the test function 
$\vp$ by saying that $\vp$ \emph{touches} $u$ at $(x_0,t_0)$ from below. 

Since, for every $\eta\in\Rn$, the inner product
$|\eta|^{-2}\langle D^2u((x_0,t_0))\eta,\eta\rangle$
lies between the smallest and the largest eigenvalue 
of $D^2u((x_0,t_0))$, this definition is equivalent to the ones in
Jin--Silvestre~\cite[Definition~2.8]{jins} and
Chen--Giga--Goto~\cite[Definition~2.1]{chengg91}.
By Lemma~2 in Manfredi--Parviainen--Rossi~\cite{manfredipr10c}, 
we may also reduce the number of test functions in the definition 
of the viscosity super/subsolutions.
More precisely, condition~(iii) in
Definition~\ref{deff-visc-sol} can equivalently be replaced by

\medskip

\begin{enumerate}
\item[(iii$'$)] for all $(x_0,t_0)\in\Theta $  and 
$\varphi\in C^2(\Theta)$, such that
  $u(x_0, t_0) = \vp(x_0, t_0)$ and $u(x, t) > \varphi(x, t)$ for 
$(x, t) \in \Theta \setm \{(x_0,t_0)\}$,
 we have 
\[ 
\begin{cases}
\varphi_t(x_0,t_0)-\Delta_p^N \varphi(x_0,t_0)\geq 0,&
		\text{if } \nabla\varphi(x_0,t_0)\neq 0,\\
\varphi_t(x_0,t_0)\ge 0, &\text{if } \nabla\varphi(x_0,t_0)=0 \text{ and } 
D^2 \vp(x_0,t_0)=0,
\end{cases}
\] 
while there is no requirement when $\nabla\varphi(x_0,t_0)=0$ 
and $D^2\vp(x_0,t_0)\ne0$.
\end{enumerate}

\medskip

This definition is the same as 
Definition~2.2 in Banerjee--Garofalo~\cite{banerjeeg15}, 
except that we do not require $u\in L^\infty$ and we allow for arbitrary
domains $\Theta\subset\R^{n+1}$, not only cylinders.
With this modification, it will be possible to obtain a full equivalence
with superparabolic functions. 
This definition is also more common in the literature.
Many of the specific supersolutions considered in this paper will
be
smooth enough to be checked by the following criterion, which may be
of independent interest.

\begin{prop} \label{prop-classical}
Assume that $u \in C^2(\Theta)$, and that for every
$(x_0,t_0) \in \Theta$ we have
\[
\begin{cases}
u_t(x_0,t_0)-\Delta_p^N u(x_0,t_0)\geq 0,&
		\text{if } \nabla u(x_0,t_0)\neq 0,\\
u_t(x_0,t_0)\ge 0, &\text{if } \nabla u(x_0,t_0)=0 
    \text{ and }  D^2 u(x_0,t_0) \ge 0,
\end{cases}
\]
then $u$ is a viscosity supersolution in $\Theta$.
\end{prop}

Note that there is no requirement if $\nabla u(x_0,t_0)=0$
and $D^2 u(x_0,t_0) \not \ge 0$.
As usual, we say that $D^2 u(x_0,t_0) \ge 0$ if it is a positive
definite or positive semidefinite quadratic form.

\begin{proof}
Let $\vp$ touch $u$ at $(x_0,t_0)$ from below as in (iii$'$).
Then 
\begin{equation} \label{eq-classical}
\nabla \phi(x_0,t_0) = \nabla u(x_0,t_0),
\
\phi_t(x_0,t_0)=u_t(x_0,t_0)
\text{ and } 
D^2 (u-\phi)(x_0,t_0) \ge 0.
\end{equation}

Assume first that  
$\nabla \phi(x_0,t_0) \ne 0$.
Let 
\[
\eta
= \frac{\nabla \phi(x_0,t_0)}{|\nabla \phi(x_0,t_0)|}
= \frac{\nabla u(x_0,t_0)}{|\nabla u(x_0,t_0)|}
\]
and let $0 \le \la_1  \le \ldots \le \la_n$ be 
the eigenvalues of the positive (semi)definite
matrix $D^2(u -\phi)(x_0,t_0)$.
Then
\begin{align*}
 \Delta_p^N u (x_0,t_0)-\Delta_p^N  \vp(x_0,t_0) 
& =\Delta (u-\phi)(x_0,t_0)
   +(p-2)\langle (D^2 (u-\phi)(x_0,t_0) \eta, \eta \rangle\\
& \ge \sum_{i=1}^n\lambda_i-\lambda_{n}
   +(p-1)\langle D^2 (u-\phi)(x_0,t_0) \eta, \eta \rangle\\
& =\sum_{i=1}^{n-1}\lambda_i
   +(p-1)\langle D^2 (u-\phi)(x_0,t_0) \eta, \eta\rangle \\
& \ge 0,
\end{align*}
and hence
\[ 
\vp_t(x_0,t_0)-\Delta_p^N \vp(x_0,t_0)
\ge u_t(x_0,t_0)-\Delta_p^N u (x_0,t_0)
\ge 0.
\] 

Assume now, on the other hand, that $\nabla \vp(x_0,t_0)=0$ and 
$D^2 \vp(x_0,t_0)=0$.
Then by \eqref{eq-classical}, $D^2 u(x_0,t_0) \ge 0$
and thus, using \eqref{eq-classical} again,
\[
      \phi_t(x_0,t_0)=u_t(x_0,t_0) \ge 0,
\]
by the assumption.
We  see that (iii$'$) is fulfilled, and therefore 
$u$ is a viscosity supersolution in $\Theta$.
\end{proof}

Note that viscosity (super)solutions are more precisely defined
than weak supersolutions used in connection with divergence-type operators,
which can be changed arbitrarily on sets of measure zero.
In fact, Proposition~\ref{prop-essliminf-reg} below implies that
if we change a viscosity supersolution at a single point, 
it will never remain to be a viscosity supersolution.

The equation \eqref{eq:normalized-p-parabolic} satisfies
some important invariance properties under multiplication, 
addition and parabolic scaling.
More precisely, 
if $u(x,t)$ is a viscosity (super/sub)solution 
and $a,\la>0$ and $b \in \R$, then so is
\[
v(x,t):=au(\la x, \la^2 t)+b.
\]

The following weak Harnack inequality for viscosity supersolutions
 can be extracted from
e.g.\ Wang~\cite[Corollary~4.14]{wang92}, 
Imbert--Silvestre~\cite[Theorem~4.15]{imberts13} 
or Jin--Silvestre~\cite[Theorem~2.1]{jins}, upon noting that
the normalized \p-Laplace operator in \eqref{eq:normalized-p-parabolic} 
satisfies the assumptions therein.
Indeed, 
\[
\Delta_p^N u = \sum_{i,j=1}^n a_{ij}(x) \bdy_{ij}u(x),
\]
where
\[
a_{ij}(x):= \de_{ij} + (p-2) \frac{\bdy_i u(x)\,\bdy_j u(x)}{|\nabla u(x)|^2}
\]
satisfy, for every $\eta\in \Rn$ with $|\eta|=1$,
\[
\sum_{i,j=1}^n a_{ij}(x) \eta_i\eta_j 
= |\eta|^2 + (p-2) \frac{\langle\nabla u(x),\eta\rangle^2}{|\nabla u(x)|^2},
\]
which clearly lies between $\la=\min\{p-1,1\}$ and $\La=\max\{p-1,1\}$.
This means that $\Delta_p^N$ can be estimated from above and below
by so-called Pucci operators, cf.~\cite[pp.\ 3--4]{jins} or 
\cite[Lemmas~3.2, 3.9 and Proposition~3.10]{wang92}.

We state the weak Harnack inequality using the space-time cylinders
\[
B_r \times (-r^2,0) \quad \text{and} \quad
B_r \times (-4r^2, -3r^2),
\]
contained in $B_{2r}\times (-4r^2,0)$, 
where $B_r=\{x\in\Rn:|x|<r\}$ are balls in $\Rn$.
These are easily obtained from the cylinders
\[
B_{1/2} \times \bigl(-\tfrac14,0\bigr), 
\quad B_{1/2} \times \bigl(-1, -\tfrac34\bigr) 
\quad \text{and} \quad B_1 \times (-1,0),
\]
used in the weak Harnack inequality in
\cite[Theorem~2.1]{jins}, by the solution-preserving
dilation $(x,t)\mapsto(2rx,(2r)^2t)$.
When $p\ge2$, the weak Harnack inequality also follows by a game-theoretic 
argument, see  Parviainen--Ruosteenoja~\cite[Theorem~4.7]{parviainenr16}.
Here and below, $\vint$ denotes the integral average, 
i.e.\ $\vint_A f \, d\mu=\int f\,d\mu/\mu(A)$.

\begin{thm}
\textup{(Weak Harnack inequality)}
 \label{thm-weak-Harnack-superpara}
Let $u$ be a non-negative viscosity supersolution
in the cylinder 
$B_{2r} \times (-4r^2,0)$, for some $r>0$. 
Then there are constants $q>0$ and $C>0$, only depending on 
$p$ and $n$, such that  
\[
\biggl(\vint_{B_r \times (-4r^2,-3r^2)} u^q\, dx\, dt \biggr) ^{1/q}
    \le C \inf_{B_r \times (-r^2,0)} u.
\]
\end{thm}

It follows from this and a covering argument that $u$ is finite a.e.
Moreover, viscosity supersolutions are lower semicontinuously
regularized in the following sense.
Banerjee--Garofalo~\cite[Proposition~3.3]{banerjeeg15}
obtained a similar result stated in a slightly weaker form.

\begin{prop} \label{prop-essliminf-reg}
Assume that $u$ is a viscosity supersolution 
in $\Theta$. 
Then for all $(x_0,t_0) \in \Theta$,
\begin{equation}   \label{eq-essliminf-reg}
    u(x_0,t_0) 
      = \liminf_{(x,t) \to (x_0,t_0)} u(x,t) 
      = \essliminf_{(x,t) \to (x_0,t_0)} u(x,t) 
      = \essliminf_{\substack{(x,t) \to (x_0,t_0) \\ t < t_0}} u(x,t).
\end{equation}
\end{prop}

\begin{proof}
We can assume that $(x_0,t_0)=(0,0)$.
Since $u$, by definition, is lower semicontinuous we directly see that
\[
    u(0,0) 
      \le  \liminf_{(x,t) \to (0,0)} u(x,t) 
      \le  \essliminf_{(x,t) \to (0,0)} u(x,t) 
      \le \essliminf_{\substack{(x,t) \to (0,0) \\ t < 0}} u(x,t).
\]
If $u(0,0)=\infty$, then there is nothing to prove.
So, without loss of generality we may assume
that  $u(0,0)=0$.
Let $a\ge 0$ be arbitrary and such that
\begin{equation}
\label{eq:low-bound}
a < \essliminf_{\substack{(x,t) \to (0,0) \\ t < 0}} u(x,t).
\end{equation}
If no such $a$ exists, then~\eqref{eq-essliminf-reg} clearly holds. 
By~\eqref{eq:low-bound}, there exists $r_0>0$ such that $u>a$ a.e.\ in 
$B_{2r_0}\times(-4r_0^2,0)$.

As $u$ is lower semicontinuous and $u(0,0)=0$, we can for any $\eps>0$
find $r\in(0,r_0)$ and $\de\in(0,r^2)$ such that $u>-\eps$  everywhere in 
\[
B_{2r}\times (\de-4r^2,\de) \subset\Theta.
\]
Since $v:=u+\eps \ge 0$ is a viscosity supersolution 
and $(0,0) \in B_r\times(\de-r^2,\de)$, 
it then follows from the weak Harnack inequality 
(Theorem~\ref{thm-weak-Harnack-superpara}) that 
\[
\biggl( \vint_{B_r\times (\de-4r^2,\de-3r^2)}     
v^q \,dx\,dt \biggr)^{1/q}
\le C v(0,0) = C \eps.
\]
As $a<u<v$ a.e.\ in 
$B_{2r_0}\times(-4r_0^2,0)$
we can conclude that $a\le C\eps$.
Letting $\eps \to 0$ shows that $a =0$, and since $a$ 
in~\eqref{eq:low-bound} was arbitrary, \eqref{eq-essliminf-reg} follows.
\end{proof}

Next we recall the definition of superparabolic functions that frequently  
appears in the nonlinear parabolic potential theory, see 
Kilpel\"ainen--Lindqvist~\cite{KiLi96} and 
Banerjee--Garofalo~\cite{banerjeeg15}.
Unless otherwise stated, $Q$ stands for the box 
$Q = (a_1, b_1)\times \ldots \times (a_n,b_n)\subset\Rn$, 
and the sets 
\[
Q_{t_1,t_2} = Q \times (t_1,t_2)
\]
are called \emph{space-time boxes}.
The \emph{parabolic boundary} of the
space-time cylinder $G_{t_1,t_2}:= G \times (t_1,t_2) \subset \R^{n+1}$
is defined as
\[
\bdyp G_{t_1,t_2}=(G \times \{t_1\}) \cup (\bdy G \times [t_1,t_2]).
\]

\begin{definition}\label{def:superparabolic}
A function $u:\Theta\rightarrow (-\infty,\infty]$
is \emph{super\-parabolic} in $\Theta$ with respect to 
\eqref{eq:normalized-p-parabolic} if
\begin{enumerate}
\renewcommand{\theenumi}{\textup{(\roman{enumi})}}%
\item $u$ is lower semicontinuous;
\item $u$ is finite in a dense subset of ${\Theta}$;
\item $u$ satisfies the following comparison principle on each 
space-time box $Q_{t_1,t_2} \Subset\Theta$:  
If $h \in C(\overline{ Q_{t_1,t_2}})$ is a viscosity solution of 
\eqref{eq:normalized-p-parabolic} in $ Q_{t_1,t_2}$  satisfying 
$h\leq u$ on $\partial_p  Q_{t_1,t_2}$, then $h\leq u$ in $ Q_{t_1,t_2}$.
\end{enumerate}

A function $u:\Theta\rightarrow [-\infty,\infty)$ is \emph{subparabolic} 
if $-u$ is superparabolic.
\end{definition}

In Banerjee--Garofalo~\cite[Definition~3.1]{banerjeeg15} they use the name
generalized super/sub\-sol\-ution rather than 
super/sub\-pa\-ra\-bolic function.
They establish the following connection between 
viscosity supersolutions and superparabolic functions 
\cite[Corollary~3.5 and Theorem~3.6]{banerjeeg15}. 
In the case of the  non-normalized \p-parabolic equation,
the corresponding equivalence
was obtained in 
Juutinen--Lindqvist--Manfredi~\cite[Theorem~2.5]{juutinenlm01}.

\begin{thm} \label{thm-BG-3.5-3.6}
In a given domain, the viscosity supersolutions 
and superparabolic functions to \eqref{eq:normalized-p-parabolic}
are the same.
\end{thm}

\begin{cor}
The weak Harnack inequality 
{\rm(}Theorem~\ref{thm-weak-Harnack-superpara}{\rm)} and the
lower semicontinuous regularity
{\rm(}Proposition~\ref{prop-essliminf-reg}{\rm)} hold also for
superparabolic functions.
\end{cor}

\begin{remark}   \label{rem-deff-differ-B-G} 
Observe that in \cite[Definition~3.1]{banerjeeg15}, 
instead of our condition (iii)  in Definition~\ref{def:superparabolic}
 they require
the comparison principle on each \emph{open cylinder} $ G_{t_1,t_2}$, 
not only on space-time boxes as here.
Thus, their definition of superparabolicity is \emph{a priori} more
restrictive than ours.

On the other hand, Theorem~\ref{thm-BG-3.5-3.6} shows that our definition of
superparabolicity is equivalent to viscosity supersolutions, 
whose definition differs from the one in \cite{banerjeeg15} only in
that we do not assume boundedness.
Since boundedness is not needed to conclude that (possibly unbounded)
viscosity supersolutions are superparabolic in the sense of 
\cite[Definition~3.1]{banerjeeg15},
it also follows that (iii) is sufficient to define the same class 
of superparabolic functions as in \cite[Definition~3.1]{banerjeeg15}. 
\end{remark}

\begin{proof}[Proof of Theorem~\ref{thm-BG-3.5-3.6}]
The only difference to \cite[Corollary~3.5]{banerjeeg15} 
is that here the viscosity super\-sol\-utions 
are not assumed to be bounded. 
However, the comparison principle for 
the viscosity super/subsolutions 
does not require this assumption, see 
Chen--Giga--Goto~\cite[Theorem~4.1]{chengg91}, and  
Giga~\cite[Corollary~3.1.5]{giga06}.
It follows that viscosity supersolutions are superparabolic in the sense
of~\cite[Definition~3.1]{banerjeeg15} and thus also in the sense of our 
Definition~\ref{def:superparabolic}.
Note that \cite[Theorem~4.1]{chengg91} is formulated in terms of
lower semicontinuously regularized (viscosity) supersolutions, but this is
provided by Proposition~\ref{prop-essliminf-reg}.

The converse direction is obtained through a counterassumption, 
see for example  Juutinen--Lindqvist--Manfredi~\cite[p.\ 704]{juutinenlm01}: 
Suppose that $u$ is a superparabolic function but that there is 
$\vp \in C^2(\Theta)$ 
as in  
Definition~\ref{deff-visc-sol}\,(iii$'$) which
touches $u$ at some $(x_0,t_0) \in\Theta$ from below and either
 \[
\varphi_t(x_0,t_0)-\Delta_p^N \varphi(x_0,t_0)<0
\quad \text{and} \quad  
\nabla\vp(x_0,t_0)\neq0,
\]
or
\[
\varphi_t(x_0,t_0)< 0, \quad \nabla\vp(x_0,t_0)=0
\quad \text{and} \quad  
D^2 \vp(x_0,t_0)=0.
\]
By continuity, 
and noting that in the latter case 
$\Delta^N_p\phi(x,t)$ 
is close to 0 whenever $\nabla \vp(x,t) \ne 0$ and $(x,t)$ is close to
$(x_0,t_0)$,
we see that there is a space-time box $Q_{t_1,t_2}\ni (x_0,t_0)$ 
such that for every $(x,t) \in Q_{t_1,t_2}$, 
either 
\[ 
\phi_t(x,t)-\Delta^N_p\phi(x,t)<0
\quad \text{and} \quad
\nabla \phi(x,t) \ne 0,
\]
or 
\[
  \phi_t(x,t)<0
\quad \text{and} \quad
\nabla \phi(x,t)=0.
\]
Thus,
$\vp$ is a continuous viscosity subsolution in  $Q_{t_1,t_2}$,
by Proposition~\ref{prop-classical}.
Since $u$ is lower semicontinuous and $\partial _p Q_{t_1,t_2}$ is compact, 
there is $\delta>0$ such that $\vp+\delta\le u$ on $\partial_p Q_{t_1,t_2}$. 
As $\phi+\de\in C^2(\Theta)$, 
using Theorem~2.6 in Banerjee--Garofalo~\cite{banerjeeg15}
we can find a viscosity
solution $h$ with continuous boundary values $\phi+\de$ on $\bdy_p Q_{t_1,t_2}$.
By the first part of the proof, $\phi+\de$ is subparabolic. Thus
property~(iii) in Definition~\ref{def:superparabolic} yields
$\vp+\delta\le h\le u$ in $Q_{t_1,t_2}$, which is a contradiction 
since $\vp(x_0,t_0)=u(x_0,t_0)$.  
\end{proof}

The following important elliptic-type comparison principle
is a slight generalization of the one in 
Banerjee--Garofalo~\cite[Lemma~3.10]{banerjeeg15}, where it was 
proved for bounded functions.
Note that if $u$ and $v$ are superparabolic, then it is easy to see that
$\min\{u,v\}$ is also superparabolic, and in particular $u_k:=\min\{u,k\}$
is superparabolic if $k \in\R$.
This fact is a special case of the pasting Lemma~\ref{lem-pasting} below,
which we however cannot yet deduce. 

\begin{thm} \label{thm-elliptic-comp}
\textup{(Elliptic-type comparison principle)}
Let $\Theta$ be a bounded open subset of  $\R^{n+1}$.
Suppose that
$u$ is super\-pa\-ra\-bol\-ic and $v$ is sub\-parabolic in $\Theta$. If
 \begin{equation} \label{eq-elliptic-comp}
   \infty \ne    \limsup_{\Theta \ni (y,s)\rightarrow (x,t)} v(y,s)\leq
    \liminf_{\Theta \ni (y,s)\rightarrow (x,t)} u(y,s) \ne -\infty
  \end{equation}
for all $(x,t) \in \partial\Theta$,
then $v\leq u$ in $\Theta$.
\end{thm}

\begin{proof}
By compactness, \eqref{eq-elliptic-comp} and semicontinuity,
$u$ is bounded from below and $v$ is bounded
from above.
Let $M=\sup_\Theta v$, $m=\inf_\Theta u$, $u_M=\min\{u,M\}$ and 
$v_m=\max\{v,m\}$. 
Then $u_M$ and $v_m$ satisfy a similar comparison on the boundary
as in \eqref{eq-elliptic-comp}.
In view of Theorem~\ref{thm-BG-3.5-3.6} 
and Remark~\ref{rem-deff-differ-B-G}, it thus follows
from Lemma~3.10 in Banerjee--Garofalo~\cite{banerjeeg15}
that $v_m \le u_M$, and hence $v \le u$,  in $\Theta$.
\end{proof}

For evolution equations, a parabolic comparison principle
is more natural since
it avoids any requirements on the future  boundary.

\begin{thm} \label{thm-parabolic-comp} 
\textup{(Parabolic comparison principle)}
Let $\Theta$ be a bounded open subset of  $\R^{n+1}$.
Suppose that $u$ is super\-parabolic and $v$ is sub\-parabolic
in $\Theta$.
Let $T \in \R$  and assume that~\eqref{eq-elliptic-comp} holds
for all  $(x,t) \in \partial\Theta$ with $t< T$.
Then $v\leq u$ in $\Thetaminus= \{(x,t) \in \Theta : t< T\}$.
\end{thm}

We will deduce the parabolic comparison principle from 
the elliptic-type comparison principle.
In order to do so we will need the following simple pasting lemma, 
which may be of independent interest.

\begin{lem} \label{lem-ext-large}
Assume that $u$ is superparabolic in 
$\Thetaminus:=\{(x,t) \in \Theta : t < T\}$.
Let $k  \in \R$. Then the function
\[
      v(x,t)=\begin{cases}
        \min\{u(x,t),k\}, & \text{if } (x,t) \in \Theta \text{ and } t < T, \\
        \min \Bigl\{\displaystyle\liminf_{\Thetaminus\ni\zeta\to(x,t)} u(\zeta),k \Bigr\}, 
                    & \text{if } (x,t) \in \Theta \text{ and } t = T, \\ 
        k, & \text{if } (x,t) \in \Theta \text{ and } t > T,
        \end{cases}
\]
is superparabolic in $\Theta$.
\end{lem}

Note that the complicated definition
for $t=T$ is needed
for $v$ to be lower semicontinuous.

\begin{proof}
By construction, $v$ is lower semicontinuous and bounded from above.
It remains to show that $v$ satisfies the comparison principle.
We therefore let $ Q_{t_1,t_2} \Subset\Theta$
be a space-time box and $h \in C(\overline{ Q_{t_1,t_2}})$ be a  
a viscosity solution in $ Q_{t_1,t_2}$ satisfying
\[
h\leq v \quad \text{on }\partial_p  Q_{t_1,t_2}.
\] 
 We first note that 
since $h \le k$ on $\partial_p  Q_{t_1,t_2}$ and the constant function $k$ is
superparabolic, it is true that $h \le k$ in $ Q_{t_1,t_2}$.
 To verify that $h\leq v$ in  $Q_{t_1,t_2}$ we let $(x,t) \in  Q_{t_1,t_2}$, 
and consider the three cases: $t>T$, $t<T$ and $t=T$
separately.

If $t>T$, then $h(x,t) \le k = v(x,t)$.
On the other hand if
$t<T$, we let $t'=\tfrac12(t+T)<T$.
Then $ Q_{t_1,t'}  \Subset \Thetaminus$ and $h \le u$ on $\bdyp  Q_{t_1,t'}$.
Together with the superparabolicity of $u$ in $\Thetaminus$, 
this yields $h(x,t) \le u(x,t)$.
As
$h \le k$, we conclude that $h(x,t) \le v(x,t)$ if $t <T$.

Finally, if $\xi=(x,T)\in Q_{t_1,t_2}$ then, since
$h\le u$ in $\Theta_{-}$,
we conclude from the continuity of $h$ 
that
\[
h(\xi)  
= \lim_{\Theta_{-}\ni\zeta\to\xi} h(\zeta) 
       \le \liminf_{\Theta_{-}\ni\zeta\to\xi} u(\zeta).
\]
As
$h \le k$, it follows that
$h(x,T) \le  v(x,T)$.
\end{proof}

\begin{proof}[Proof of Theorem~\ref{thm-parabolic-comp}]
Let $(x_0,t_0) \in \Theta$ with $t_0 <T$, and set $T'=\tfrac{1}{2}(t_0+T)$.
Then the lower semicontinuity of $u$ and upper semicontinuity of $v$, 
together with 
\eqref{eq-elliptic-comp} show that $u$ is bounded from below
and $v$ is bounded from above in 
\[\Thetatminus:=\{(x,t)\in \Theta : t \le T'\}.
\]
Let $m=\inf_{\Thetatminus} u$, $M=\sup_{\Thetatminus} v$,
\[
      \ut=\begin{cases}
        \min\{v,M\} & \text{in } \Thetatminus, \\
        M & \text{in } \Theta \setm \Thetatminus,
        \end{cases}
        \quad \text{and} \quad
      \vt=\begin{cases}
        \max\{v,m\} & \text{in } \Thetatminus, \\
        m & \text{in } \Theta \setm \Thetatminus.
        \end{cases}
\]
By Lemma~\ref{lem-ext-large}, $\ut$ is superparabolic and $\vt$ is 
subparabolic in $\Theta$.
Now $\ut$ and $\vt$ satisfy the assumptions for the elliptic comparison
principle (Theorem~\ref{thm-elliptic-comp}) in $\Theta$,
and thus $\vt \le \ut$ in $\Theta$.
Hence $v(x_0,t_0) \le \vt(x_0,t_0)\le \ut(x_0,t_0) \le u(x_0,t_0)$.
\end{proof}

Having established the parabolic comparison principle 
(Theorem~\ref{thm-parabolic-comp}), we can obtain the following
generalization of Lemma~\ref{lem-ext-large},
which is useful when constructing new superparabolic
functions.

\begin{lem} \label{lem-pasting}
\textup{(Pasting lemma)}
Let $U \subset \Theta$ be open. 
Also let $u$ and $v$ be super\-pa\-ra\-bo\-lic in $\Theta$ and $U$,
respectively,
and let
\[
    w=\begin{cases}
     \min\{u,v\} & \text{in } U, \\
     u & \text{in } \Theta \setm U. \\
    \end{cases}
\] 
If $w$ is lower semicontinuous, then $w$ is superparabolic in $\Theta$.
\end{lem}

\begin{proof}
Since $-\infty < w \le u$, we see that 
$w$ is finite in a dense subset of $\Theta$,
and we only have to obtain the comparison principle.
Therefore, let $Q_{t_1,t_2}  \Subset \Theta$ be a space-time box, 
and $h \in C(\overline{  Q_{t_1,t_2}})$ 
be a viscosity solution in $  Q_{t_1,t_2}$ such that 
\begin{equation} \label{eq-bdyp}
h \le w \quad \text{on }\bdyp Q_{t_1,t_2}.
\end{equation}
Since $h \le u$ on $\bdyp Q_{t_1,t_2}$ and $u$ is superparabolic, 
we directly have that 
\begin{equation}
\label{eq:h-le-u}
h \le u\quad  \text{in }  Q_{t_1,t_2}.
\end{equation}
To complete the proof we show  that 
\[
h \le v \quad \text{in }Q_{t_1,t_2} \cap U.
\]
To this end, we intend to use the parabolic comparison principle 
for $Q_{t_1,t_2} \cap U$ after verifying that $h(x,t) \le v(x,t)$ for 
$(x,t) \in \partial (Q_{t_1,t_2} \cap U)$ with $ t< t_2$. 
There are two cases: either $(x,t) \in U$ or $(x,t) \notin U$. 

First, suppose  that $(x,t) \in U$, then $(x,t) \in \bdyp   Q_{t_1,t_2}$ and thus
by the lower semicontinuity  of $v$,
\[
    \liminf_{Q_{t_1,t_2} \cap U \ni (y,s) \to (x,t)} v(y,s) \ge v(x,t) \ge w(x,t) \ge h(x,t),
\]
where the last inequality  follows from \eqref{eq-bdyp}. 
On the other hand, if $(x,t) \notin U$, then 
by the lower semicontinuity of $w$, 
\[
    \liminf_{Q_{t_1,t_2} \cap U \ni (y,s) \to (x,t)} v(y,s) \ge \liminf_{Q_{t_1,t_2} \cap U \ni (y,s) \to (x,t)} w(y,s) 
    \ge w(x,t) = u(x,t) \ge h(x,t),
\]
where the last  inequality follows from  \eqref{eq-bdyp} 
or \eqref{eq:h-le-u}, depending on whether 
$(x,t) \in \bdyp   Q_{t_1,t_2}$ or 
$(x,t) \in  Q_{t_1,t_2}$.
Hence, the parabolic comparison principle (Theorem~\ref{thm-parabolic-comp}) 
 shows that $h \le v$ in $Q_{t_1,t_2} \cap U$.
Together with \eqref{eq:h-le-u} this shows that
 $h \le w$ in $  Q_{t_1,t_2}$.
\end{proof}

The strong minimum principle for superparabolic functions
will be an important tool for us.
In the statement, we will use  polygonal paths. 
A  \emph{polygonal path}
is a continuous and piecewise linear function   
$\ga:[0,1]\to\Theta$.

\begin{theorem}   \label{thm-strong-min-princ}
\textup{(Strong minimum principle)}
Let $u \ge 0$ be superparabolic in $\Theta$,
$\xi_0  \in \Theta$ and let $\La$ be the set of all points $\xi \in \Theta$ 
such that there is a polygonal path $\ga:[0,1] \to \Theta$
from $\xi =\ga(0)$ to $\xi_0 =\ga(1)$ 
along which the time variable is strictly increasing.
If $u(\xi_0)= 0$,  then $u \equiv  0 $ in $\bar{\Lambda} \cap \Theta$.
\end{theorem}

\begin{proof}
First, let $\xi \in \Lambda$,
and let $\ga: [0,1] \to \Theta$ be a polygonal path from $\xi=\ga(0)$ to 
$\xi_0= \ga(1)$
along which the time variable is strictly increasing.
Also let
\[
\sigma =\inf \{s\in [0,1] : u(\ga(s))=0\}.
\]
By the lower semicontinuity of $u$, we see that $u(\ga(\sigma))=0$.
For simplicity we assume that $\ga(\sigma)=(0,0)$.

If $\sigma>0$, then 
there is $0<s< \sigma$ and $r>0$ such that
\[
B_{2r}\times(-4r^2,0) \subset \Theta \quad \text{and} \quad
   \ga(s) \in B_r \times  (-4r^2,-3r^2).
\]
It then follows from the weak Harnack inequality
(Theorem~\ref{thm-weak-Harnack-superpara}) 
together with  Proposition~\ref{prop-essliminf-reg} that
\begin{align*}
\biggl(\vint_{B_r\times (-4r^2,-3r^2)} u^q\, dx\, dt \biggr) ^{1/q}
     &\le C \inf_{B_r\times (-r^2,0)} u \\
     &\le C \essliminf_{\substack{(x,t) \to (0,0) \\ t < 0}} 
         u( x,t) =  C u (0,0) =0.
\end{align*}
Thus $u =0$ a.e.\ in $B_r\times (-4r^2,-3r^2)$. 
Since $u \ge 0$ is lower semicontinuous it must be identically
$0$  therein.
In particular $u(\ga(s))=0$, but this contradicts the fact that $s < \sigma$.
Hence $\sigma =0$ and $u(\xi)=u(\ga(\sigma))=0$.

Finally, as $u \ge 0$ is lower semicontinuous it follows that
$u \equiv 0$ in $\bar{\Lambda} \cap \Theta$.
\end{proof}

\section{Perron solutions and boundary regularity}

\emph{In this section we assume that $\Theta \subset \R^{n+1}$ is a bounded open set.}

\medskip

Perhaps the most general method to solve the Dirichlet problem
in arbitrary bounded domains is the Perron method.
For us it will be enough to consider Perron solutions for
bounded functions, so for simplicity we restrict ourselves
to this case throughout  the rest of this paper.

\begin{definition}   \label{def-Perron}
Given a bounded function $f \colon \bdy \Theta \to \R$,
let the upper class $\UU_f$ be the set of all
superparabolic  functions $u$ on $\Theta$ such that
\[ 
    \liminf_{\Theta \ni \eta \to \xi} u(\eta) \ge f(\xi) \quad \text{for all }
    \xi \in \bdy \Theta.
\] 
Define the \emph{upper Perron solution} of $f$  by
\[
    \uP_{\Theta} f (\xi) = \inf_{u \in \UU_f}  u(\xi), \quad \xi \in \Theta.
\]
Similarly, let the lower class $\LL_f$ be the set of all subparabolic functions  
$ {v}$ on $\Theta$
 such that
\[ 
    \limsup_{\Theta \ni \eta \to \xi}  v(\eta) \le f(\xi) \quad \text{for all }
    \xi \in \bdy \Theta 
\] 
and define the \emph{lower Perron solution}  of $f$ by
\[
  \lP_{\Theta} f(\xi) = \sup_{v \in \LL_f}   v(\xi), \quad \xi \in \Theta.
\]
\end{definition}

It follows directly from the elliptic comparison
principle (Theorem~\ref{thm-elliptic-comp}) that
we always have $\lP f  \le \uP f$.
Moreover, $\lP f$ and $\uP f$ are viscosity solutions,
see Theorem~3.12 in Banerjee--Garofalo~\cite{banerjeeg15}  
and also Section~2.4 in Giga~\cite{giga06}.
When the Perron solution is taken with respect to $\Theta$
we often drop  $\Theta$ from the notation.

\begin{definition}
\label{def:regular}
A boundary point $\xi_0\in \partial \Theta$ is 
\emph{regular} with respect to $\Theta$  if
\[
        \lim_{\Theta \ni \xi \to \xi_0} \uP f(\xi)=f(\xi_0)
\quad  \text{whenever $f \in C(\partial \Theta)$.}
\]
\end{definition}

Since $\lP f = - \uP (-f)$, boundary regularity can equivalently
be formulated using lower Perron solutions. 

\begin{deff}
A function $w$ is a barrier in $\Theta$ at the point $\xi_0 \in \bdy \Theta$
if
\begin{enumerate}
\renewcommand{\theenumi}{\textup{(\roman{enumi})}}%
\item 
$w$ is a positive superparabolic function in $\Theta$;
\item
$\lim_{\Theta \ni \zeta \to \xi_0} w(\zeta) =0$;
\item \label{barrier-ii}
$\liminf_{\Theta \ni \zeta \to \xi} w(\zeta)  >0$
for every $\xi \in \bdy \Theta \setm \{\xi_0\}$.
\end{enumerate}
\end{deff}

Banerjee--Garofalo~\cite[Section~4]{banerjeeg15}
obtained a number of  results 
about boundary regularity which we summarize as follows.
(Part \ref{k-Thetaminus} follows from \ref{k-local} together
with Proposition~4.7 and Theorem~4.8 in \cite{banerjeeg15}.)

\begin{thm} \label{thm-barrier-and-more}
Let $\xi_0=(x_0,t_0) \in \bdy \Theta$ and let 
$\Theta_{-}=\{(x,t) \in \Theta : t < t_0\}$.
\begin{enumerate}
\item \label{k-barrier-char}
$\xi_0$ is regular if and only if there is a barrier at $\xi_0$.
\item \label{k-local}
Regularity is a local property, i.e.\ if $U$ is  an open
neighbourhood of $\xi_0$,
then $\xi_0$ is regular with respect to $\Theta$ if and only if it is regular
with respect to $\Theta \cap U$.
\item \label{k-Thetaminus}
$\xi_0$ is regular with respect to $\Theta$ if and only if
either $\xi_0 \notin \bdy \Thetaminus$ or
$\xi_0$ is regular with respect to $\Thetaminus$.
\item 
If $\xi_0$ is regular and $f: \bdy \Theta \to \R$ 
is a bounded function which is continuous at $\xi_0$,
then
\[
     \lim_{\Theta \ni \xi \to \xi_0} \lP f(\xi) 
   = \lim_{\Theta \ni \xi \to \xi_0} \uP f(\xi) 
   = f(\xi_0).
\]
\end{enumerate}
\end{thm}

In particular, part \ref{k-Thetaminus} implies that a first point
is always regular, because in this case  $\Thetaminus= \emptyset$.
Another important consequence of the barrier characterization
is the following 
\emph{restriction} property.

\begin{prop} \label{prop-restrict}
Let $\xi_0 \in \bdy \Theta$ and let $U \subset \Theta$ be open and such
that $\xi_0 \in \bdy U$.
If $\xi_0$ is regular with respect to $\Theta$,
then $\xi_0$ is regular with respect to $U$.
\end{prop}

\begin{proof}
By Theorem~\ref{thm-barrier-and-more}\,\ref{k-barrier-char} there is a 
barrier $w$ in $\Theta$ at $\xi_0$. 
As $w$ is lower semicontinuous and positive it follows
directly that $w|_U$ is a barrier with respect to $U$.
Thus Theorem~\ref{thm-barrier-and-more}\,\ref{k-barrier-char} 
implies that $\xi_0$ is a regular boundary point with respect to $U$.
\end{proof}

\section{The tusk condition}

\begin{definition}
\label{def-tusk}
A \emph{tusk} at $\xi_0=(0,0)\in\partial\Theta$ is a set in $\Rno$ of the form
\[
V:= \{(x,t)\in\R^{n+1}: -T<t<0 \text{ and } |x-(-t)^{1/2} \xh|^2 < R^2(-t)\},
\]
for some $ \xh\in\Rn$ 
and with positive constants $R$ and $T$, see Figure~\ref{Fig:1}.
We say that $\Theta$ satisfies the \emph{tusk condition} 
at $\xi_0=(0,0)\in\partial\Theta$ if there is a tusk $V$  {at $\xi_0$}
with  $V\cap \Theta =\emptyset$. 

At points $(0,0)\neq \xi_0\in\bdry\Theta$, the definition is analogous 
except that we use translations of $V$.
\end{definition} 

It is well known that if $\xi_0$ satisfies the tusk condition
then $\xi_0$ is regular for the heat equation,
see Effros--Kazdan~\cite{effrosk71}, which refers to $\xi_0$ as being 
\emph{parabolically touchable}, and 
Lieberman~\cite{lieberman89}. 
We extend this result to the normalized \p-parabolic equation. 
Compared to \cite{effrosk71}, we do not establish a counterpart 
of their Lemma~1, but use the iterative argument directly
together with the parabolic comparison principle.
We also improve on their result (also for the heat equation) by
showing H\"older continuity.

To start with, we prove an auxiliary exterior ball condition.
We let $B(\zeta,R)=\{\xi \in \R^{n+1} : |\zeta - \xi| < R\}$ denote
a ball in $\R^{n+1}$.

\begin{lem}
\label{lem:exterior-ball-condition-first}
\textup{(Exterior ball condition, preliminary version)}
Let $\xi_0=( x_0,t_0)\in\partial\Theta$.
Suppose that there exists a  ball $B=B(\xi_1,R_1)$, $\xi_1=(x_1,t_1)$,
such that $B \cap \Theta = \emptyset$ and
$\xi_0 \in \bdy B \cap \bdy \Theta$.
If $x_1\ne x_0$, or if $\xi_0$ is the north pole
of $B$ \textup{(}i.e.\ $\xi_1=(x_0,t_0-R_1)$\textup{)}
and the additional radius condition $R_1>n+p-2$ is satisfied,
then $\xi_0$ is regular with respect to $\Theta$.
\end{lem}

In Proposition~\ref{prop:exterior-ball-condition}
we will remove the above restriction 
on the radius in the north pole case.

\begin{proof} 
For simplicity, we assume that $\xi_0=(0,0)$.
By choosing a smaller ball, if necessary, we may  without loss of generality 
assume that 
$\bdy B \cap \bdy \Theta=\{ \xi_0\}$. 
For $\xi=(x,t)$ define
\[
w(\xi) =e^{-j R_1^2}-e^{-j R^2},
\]
where $R=|\xi-\xi_1|$ 
and  $R_1=|\xi_1|$, while
$j$ will be chosen later.
Then $w>0$ in $\overline{\Theta}\setm\{\xi_0\}$ and 
$\lim_{\xi\to\xi_0} w(\xi)=0$.
Elementary calculations show that
\begin{align*}
w_t(x,t) & = 2j e^{-j R^2} (t-t_1), \\
\nabla w(x,t) & = 2je^{-jR^2}(x-x_1), \\
\Delta_p w(x,t) &= (2j)^{p-1} |x-x_1|^{p-2} e^{-j(p-1)R^2} [n+p-2-2j(p-1)|x-x_1|^2],
\end{align*}
and, provided that $\nabla w(x,t) \ne 0$,
\[
\Delta^N_p w(x,t) 
= 2j e^{-jR^2} [n+p-2 - 2j(p-1)|x-x_1|^2 ].
\]
Thus, still assuming that $\nabla w(x,t) \ne 0$,
\begin{equation}    \label{eq-calc-for-w}
\Delta^N_p w(x,t) -  w_t (x,t)
= 2j e^{-jR^2} [n+p-2 - 2j(p-1)|x-x_1|^2 - (t-t_1)].
\end{equation}
To show that $w$ is superparabolic, we need to show that the last bracket
is nonpositive.
Since regularity is a local property by 
Theorem~\ref{thm-barrier-and-more}\,\ref{k-local}, 
we may restrict our considerations to a small neighbourhood of $\xi_0$.

If $x_1\ne0$   then, 
in view of Theorem~\ref{thm-barrier-and-more}\,\ref{k-local}, we can
assume that $(x,t) \in \Theta$ satisfy $|x|,|t|<\de:=\tfrac12|x_1|$.
In particular, $|x-x_1|>\de$, $\nabla w(x,t) \ne 0$
and $t_1-t<t_1+\de$. 
Hence we can choose $j$ so that the bracket in~\eqref{eq-calc-for-w}
is nonpositive and  thus
$\Delta^N_p w(x,t) - w_t(x,t) \le0$ for all such $x$ and $t$.
By Proposition~\ref{prop-classical}, we get that $w$ is superparabolic.

If, on the other hand, $x_1=0$ and $t_1=-R_1$, then 
we can assume that $t>n+p-2-R_1$ (which is negative by assumption)
 whenever $(x,t) \in \Theta$.
In particular, $w_t(x,t)>0$.
Moreover, if $\nabla w(x,t) \ne 0$, then
\[
\Delta^N_p w(x,t) -  w_t(x,t) 
\le 2j e^{-jR^2} [n+p-2-R_1-t] <0.
\]
Hence $w$ is superparabolic, by Proposition~\ref{prop-classical}.
\end{proof}

\begin{figure}[t]
\begin{center}
\includegraphics[width=.5\textwidth]{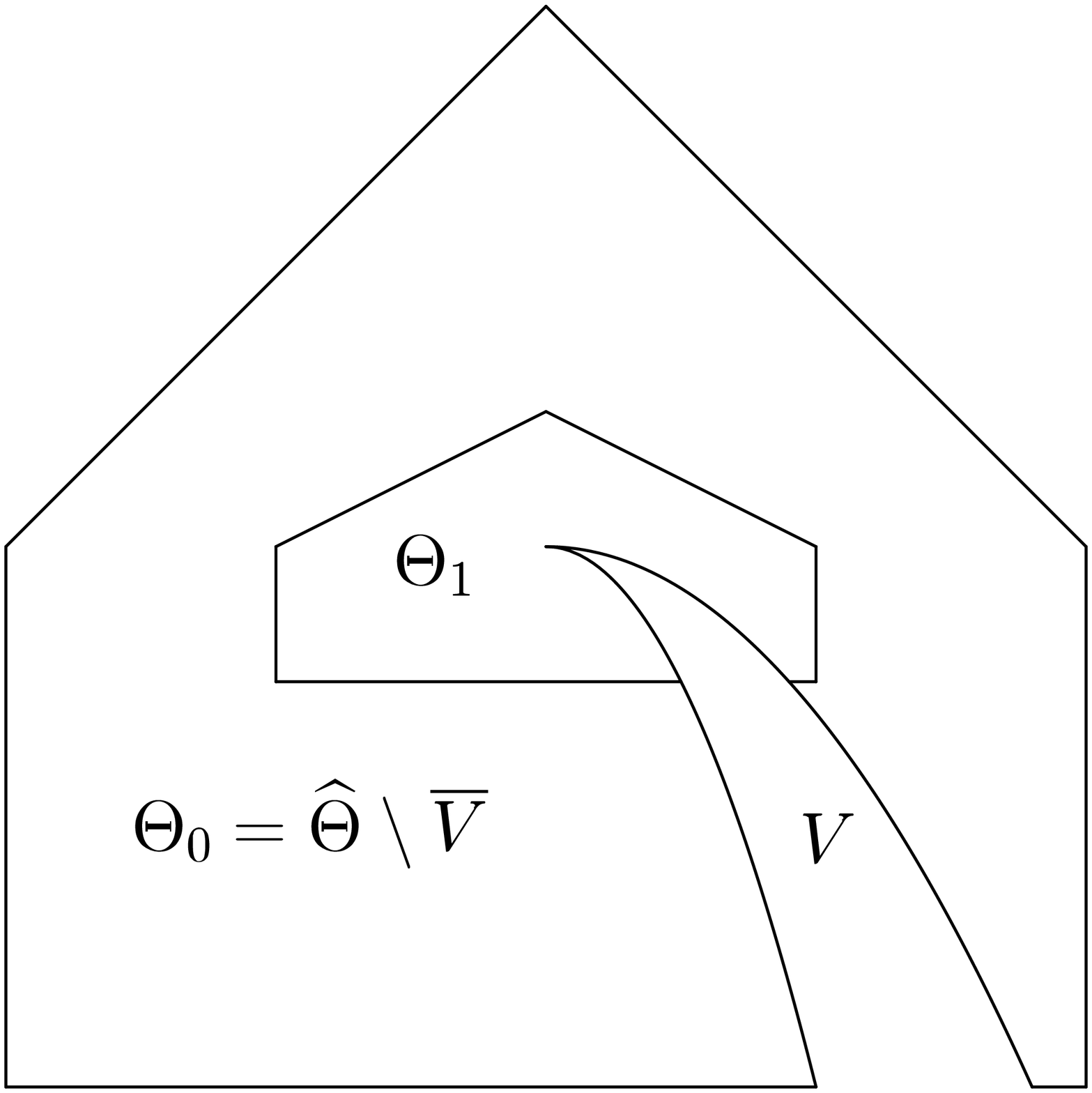}
\end{center}
\caption{The domains in Lemma~\ref{lem-tusk-house-barrier}.}\label{Fig:1}
\end{figure}

\begin{lem}   \label{lem-tusk-house-barrier}
Let $V$ be a tusk at $\xi_0=(0,0)$, determined by $T=1$, $R$ and $\xh$.
Then $\xi_0$ is regular with respect to the domain $\Theta_0=\Thetah\setm\clV$,
where
\[
\Thetah = ( B_{R_0}\times (-1,0])\cup \{(x,t)\in\R^{n+1}: |x|<R_0(1-t)
   \text{ and } 0 \le t<1\}
\]
for some $R_0>|\xh|+R$, see Figure~\ref{Fig:1}. 
Moreover, the  viscosity solution
$u:=\uP_{\Theta_0} f$, with
\[
f(x,t) = \begin{cases}
        -t, & \text{if } (x,t) \in \bdry\Theta_0 \cap \bdry V, \\ 
         1, & \text{if } (x,t) \in \bdry\Theta_0 \setm \bdry V,
        \end{cases}
\]
is a positive continuous barrier in $\Theta_0$, 
which is H\"older continuous at $\xi_0$
and continuously attains its boundary values $f$ everywhere on $\bdry\Theta_0$.
\end{lem}

Here and below, we mean H\"older continuity with respect
to parabolic scaling, i.e.\ $g$ is \emph{H\"older continuous}
at $(0,0)$ with H\"older exponent $\beta$ if
\[
   |g(x,t)-g(0,0)| \le C(|x|+|t|^{1/2})^\beta.
\]

\begin{proof}
Continuity of $u$ within $\Theta_0$ is clear since it is 
a viscosity solution therein. 
By the exterior ball condition (Lemma~\ref{lem:exterior-ball-condition-first}),
all $(x,t)\in\bdry\Theta_0\setm\{\xi_0\}$ are regular and hence
\begin{equation} \label{eq-reg-lim}
\lim_{\xi\to(x,t)} u(\xi) = f(x,t) >0 
\quad \text{for all } (x,t)\in\bdry\Theta_0\setm \{\xi_0\}.
\end{equation}
From the strong minimum principle 
(Theorem~\ref{thm-strong-min-princ}), 
together with~\eqref{eq-reg-lim}
and the fact that viscosity solutions are preserved
under multiplication and 
addition by constants,
we conclude that $0<u<1$ in $\Theta_0$.

To show that $u$ is a barrier,
it suffices to show that $\lim_{\Theta_0\ni\xi\to\xi_0} u(\xi) = 0$.
To this end, let $v(x,t)=u(2x,4t)$ and 
$\Theta_k=\Thetah_k\setm \clV$, where
\begin{equation}   \label{eq-def-Thetah-k}
\Thetah_k = \{(x,t)\in\R^{n+1}: (2^kx,4^kt)\in\Thetah\}, \quad k=0,1,\ldots,
\end{equation}
see Figure~\ref{Fig:1}.
Note that $\Theta_{k+1}\subset\Theta_k$, $k=0,1,\ldots$,
with identical boundaries near $\xi_0$,
and that 
\[
  K:= \overline{\bdry\Theta_1\setm \bdry V} \subset \overline{\Theta}_0
\]
is compact.
Hence, by continuity and \eqref{eq-reg-lim}, we see that
$\al_1:=\sup_K u <1$.
At the same time,
\[
   v=1\quad \text{on }K
\]
and $v$ attains the boundary values 
\[
v(x,t)=-4t=4u(x,t)
\quad \text{on }\bdry\Theta_1\cap\bdry V \setm \{\xi_0\}.
\]
The parabolic comparison principle
(Theorem~\ref{thm-parabolic-comp}), applied to
$\Theta^-_1 := \{(x,t)\in\Theta_1:t<0\}$,
implies that 
\begin{equation}   \label{eq-u-le-al-v}
u\le\al v \quad \text{in }\Theta^-_1,
\end{equation}
where $\al=\max\bigl\{\al_1,\tfrac14\bigr\}$.
Thus, if 
\[
A:=\limsup_{\Theta^-_1\ni\xi\to\xi_0} u(\xi),
\] 
then $0\le A\le\al A$, from which it follows that $A=0$.
At the same time, since $0<u<1$ in $\Theta_0$, we conclude from the continuity
of $u$ in $\Theta_0$ that 
\[
\liminf_{\Theta^-_1\ni\xi\to(x,0)} u(\xi) =u(x,0)>0
\]
whenever $x\ne0$. 
Thus, $u$ is a barrier for $\Theta^-_1$ at $\xi_0$ and $\xi_0$ is regular
for $\Theta^-_1$.
Theorem~\ref{thm-barrier-and-more} then implies that $\xi_0$ is regular
for $\Theta_0$ as well.
In particular, this means that $\lim_{\Theta_0\ni\xi\to\xi_0} u(\xi) = 0$.

We shall now show that $u$ is H\"older continuous at $\xi_0$.
From the first part of the proof we see that 
\[
\limsup_{\Theta_1\ni\xi\to(x,t)} u(\xi) \le \liminf_{\Theta_1\ni\xi\to(x,t)} \al v(\xi)
\]
for all $(x,t)\in\bdry\Theta_1$.
The elliptic comparison principle (Theorem~\ref{thm-elliptic-comp}) 
then implies that $u\le\al v$ in $\Theta_1$.
An iteration of this inequality then gives for 
$(x,t)\in\Theta_k\setm\Theta_{k+1}$ that
\begin{equation}   \label{eq-alpha-k}
u(x,t)\le \al u(2x,4t) \le \cdots \le \al^k u(2^kx,4^kt) \le \al^k
\le C (|x|+|t|^{1/2})^{\beta},
\end{equation}
where $\beta=-\log\al/{\log2}>0$. 
Since this holds for all $k=1,2,\ldots$, we see that $u$ is H\"older 
continuous at $\xi_0$.
\end{proof}

\begin{thm}
\label{thm-tusk}
\textup{(The tusk condition)}
If $\Theta$ satisfies the tusk condition at $\xi_0$ then $\xi_0$ is 
regular.
If moreover, $f:\bdry\Theta\to\R$ is bounded and H\"older continuous 
at $\xi_0$ then so is $\uP f$.
\end{thm}

It follows from the proof below that if $f$ is bounded
and H\"older continuous 
at $\xi_0$ with H\"older exponent $\ga>0$,
and $\ga$ is small enough, then $\uP f$ is H\"older continuous 
at $\xi_0$ with H\"older exponent $\tfrac12\ga$.
How small $\ga$ has to be depends on the tusk.

In fact, replacing the scaling $(2^kx,4^kt)$ in~\eqref{eq-def-Thetah-k} 
  by $(b^kx,b^kt)$ with any $b>1$, and $l=k$, $k+1$
by $l=l_k,l_k+1,\ldots,l_{k+1}$, where
    $l_k=\lceil -\ga k\log b/\log\al \rceil$,
  in the proof below,
makes it possible to obtain H\"older continuity at $\xi_0$ with any exponent
$\beta<\ga$, at the cost of an increasing constant $C''$ 
in~\eqref{eq-Pf-C''}.
We leave the details to the interested reader.

\begin{proof}
We can assume that $\xi_0=(0,0)$.
The regularity of $\xi_0$ follows from Lemma~\ref{lem-tusk-house-barrier}
by means of the restriction property
(Proposition~\ref{prop-restrict}) and the fact that regularity is 
a local property, by Theorem~\ref{thm-barrier-and-more}.

To prove the H\"older regularity,
assume that $f:\bdry\Theta\to\R$ is H\"older continuous near $\xi_0$
with H\"older exponent $\ga$
and that $|f|\le M$ on $\bdry\Theta$. 
We can also assume that $f(0,0)=0$.
Let  $k\ge0$ be arbitrary, 
but such that $\Theta\cap \Thetah_k \subset \Theta_k$
and 
\[
  |f(x,t)|\le C (|x|+|t|^{1/2})^{\ga}
   \quad \text{on }\bdry\Theta\cap \Thetah_k,
\]
where $\Thetah_k$ and  $\Theta_k$ are as in~\eqref{eq-def-Thetah-k}.
Let $u$ be the barrier  from Lemma~\ref{lem-tusk-house-barrier}
and $\alp$ be the corresponding constant from~\eqref{eq-u-le-al-v}.
Set $u_k(x,t)=u(2^kx,4^kt)$ in $\Theta_k$.
Extend $u_k$ by 1 to $\Theta\setm \Theta_k$
and then by continuity to $\bdy \Theta$.
By the pasting Lemma~\ref{lem-pasting}, $u_k$ is superparabolic in $\Theta$, 
and provides us with a H\"older continuous barrier therein, 
in view of Lemma~\ref{lem-tusk-house-barrier}.

Since $u_k=1$ on $\bdry\Theta\setm \Thetah_k$, we have
$f \le C'2^{-\ga k} + M u_k$ everywhere on $\bdry\Theta$. 
Hence, by the definition of Perron solutions,
$\uP f \le C'2^{-\ga k} + M u_k$ in $\Theta$.
Thus, for $l\ge1$ and
$(2^kx,4^kt)\in \Theta \cap (\Thetah_{l} \setm \Thetah_{l+1})$
we conclude from~\eqref{eq-alpha-k} that
\[
\uP f(x,t) \le C'2^{-\ga k} + M u(2^kx,4^kt) \le C'2^{-\ga k} + M \al^{l}.
\]
In particular, with $l=k$ and $l=k+1$, i.e.\ for
$(x,t)\in \Theta \cap (\Thetah_{2k} \setm \Thetah_{2(k+1)})$,
\begin{equation}   \label{eq-Pf-C''}
\uP f(x,t) \le C'' (|x|+|t|^{1/2})^{\beta},
\end{equation}
where $\beta=\min \{\ga/2,-\log\al/2\log 2\}>0$. 
Applying the same argument to $-f$ and by considering all sufficiently 
large $k$ shows that $\uP f$ is H\"older continuous at $\xi_0$.
\end{proof}

As a direct consequence of the tusk condition we can now
deduce the following ``wedge'' condition for cylinders.

\begin{cor} \label{cor-cyl-p<2}
Let $G \subset \R^n$ be open and $\Theta=G_{t_1,t_2}$.
Let $(x_0,t_0) \in \bdy G \times  [t_1,t_2]$ be a  
point on the lateral 
boundary.
Assume that there is  $a>0$ and a vector $y \in \R^n$  such
that the cone
\begin{equation}   \label{eq-cone-at-G}
   \{x \in\R^n: (x-x_0) \cdot y > a |x-x_0|\} 
\end{equation}
belongs to the complement $\R^n \setm G$ of $G$.
Then $(x_0,t_0)$ is a regular boundary point for $\Theta$.
\end{cor}

\begin{remark}    \label{rem-cyl-p<2} 
The proofs of Lemma~\ref{lem-tusk-house-barrier} and Theorem~\ref{thm-tusk}
reveal that the tusk $V$ therein can be replaced by the following union
of geometrically spaced ellipses,
\[
E_k=\biggl\{(x,t)\in\R^{n+1}: \biggl( \frac{|x-x_k|}{a_k} \biggr)^2 
              + \biggl( \frac{t-t_k}{b_k} \biggr)^2 < 1\biggr\},
\quad k=1,2,\ldots,
\]
where $x_k=q^k\hat{x}$, $a_k=aq^k$, $b_k=bq^{2k}$ and $t_k=-cq^{2k}$
for some $a,b,c>0$, $\hat{x} \in \R^n$ and $0<q<1$.
More precisely, assuming that 
$\Theta\cap E_k =\emptyset$, $k=1,2,\ldots$,
we have that $\xi_0=(0,0)$ is regular for $\Theta$, with a H\"older
continuous barrier.
Moreover, H\"older continuity of the boundary data $f$ at $\xi_0$ implies
H\"older continuity of $\uP f$ at $\xi_0$.

Similarly,
the ``wedge'' condition
\eqref{eq-cone-at-G}
in Corollary~\ref{cor-cyl-p<2} can be replaced by the requirement that
\[
G \cap B_k = \emptyset, \quad k=1,2,\ldots,
\]
where 
$x_k$ and $a_k$ are as above, and
$B_k=\{x \in \R^n : |x-x_k|<a_k\}$. 
\end{remark}

For the non-normalized \p-parabolic equation {(with $p>1$)} 
it was shown in Kilpel\"ainen--Lindqvist~\cite{KiLi96}
and Bj\"orn--Bj\"orn--Gianazza--Par\-vi\-ain\-en~\cite[Theorem~3.9]{BBGP}
that a point $(x_0,t_0)$ on the lateral boundary 
of a cylinder $G_{t_1,t_2}\subset\R^{n+1}$ is regular if and only if
$x_0$  is regular for \p-harmonic functions
with respect to $G\subset\R^n$.
The main result in 
Banerjee--Garofalo~\cite[Theorem~1.1]{banerjeeg15}
says that the same equivalence  holds for the normalized \p-parabolic equation
provided that $p \ge 2$.
However, due to the  
power $2-p$ in $\Delta_p^N u = |\nabla u|^{2-p} \Delta_p u$, 
which leads to the singular right-hand side in 
$\Delta_p u=-|\nabla u|^{p-2}$, they did not cover the case $p<2$.
Corollary~\ref{cor-cyl-p<2} and Remark~\ref{rem-cyl-p<2} are currently the
best known results about boundary regularity for cylinders when $p<2$.
Note, however, that the necessity part of the proof 
of~\cite[Theorem~1.1]{banerjeeg15}
(i.e.\ from the regularity of $(x_0,t_0)$ to the regularity of $x_0$)
holds true also for $p<2$.

We end this section by
deducing the full exterior ball condition.

\begin{prop}
\textup{(Exterior ball condition)}
\label{prop:exterior-ball-condition}
Let $\xi_0=(x_0,t_0)\in\partial\Theta$.
Suppose that there exists a  ball $B=B(\xi_1,R_1)$, $\xi_1=(x_1,t_1)$,
such that $B \cap \Theta = \emptyset$ and
$\xi_0 \in \bdy B \cap \bdy \Theta$.
If $x_1\ne  x_0 $, or if $\xi_0$ is the north pole
of $B$ \textup{(}i.e.\ $\xi_1=(x_0,t_0-R_1)$\textup{)},
then $\xi_0$ is regular with respect to $\Theta$.
\end{prop}

Note that this result is a direct corollary of the tusk
condition (Theorem~\ref{thm-tusk}), since the exterior ball condition
is always a stronger requirement than the tusk condition.
Nevertheless, we only need to directly appeal to the tusk condition
for the north pole case.

\begin{proof}
Apart from the north pole case this follows from
Lemma~\ref{lem:exterior-ball-condition-first},
while the north pole case follows directly
from the tusk condition (Theorem~\ref{thm-tusk}).
\end{proof}

Note that the well-known irregularity of nonlateral last points in cylinders 
shows that an exterior ball touching at the south pole
does not guarantee regularity,
which leads us directly into the topic of the next section.

\section{The \texorpdfstring{Petrovski\u\i}{Petrovskii}   criterion}
\label{sect-Petr}

In this section, we consider the regularity of the last point of
a domain.
Nonlateral last points in cylinders are known to be irregular.
On the other hand, last points of paraboloids are regular
by e.g.\ the tusk condition.
The idea in the Petrovski\u\i\   criterion is to
find a sharper regularity condition for the shape of the domain 
near a last point. 
This condition has also interesting
consequences.  Just as for the heat equation, Theorem~\ref{thm-Petr}, 
together with a simple
scaling argument, shows that regularity of a boundary point for
the multiplied equation
\[
 a  u_t -\Delta_p^N u=0, 
\quad \text{with } a >0,
\]
depends on $a$.

\begin{thm} \label{thm-Petr}
Let 
\[
   \Theta :=\{(x,t) : |x|^2 <  A|t| \log {\lt}
   \text{ and } -\tfrac13 < t < 0\},
\]
where $A>0$.
Then $\xi_0=(0,0)$ is regular if and only if
$A \le 4(p-1)$.
\end{thm}

In view of  Theorem~\ref{thm-barrier-and-more}\,\ref{k-local} the constant
$-\tfrac13$ can be replaced by any other negative constant,
but here it has been chosen so that $\log {\lt} > 0$ for all such~$t$.

\begin{proof}
We first consider the case when $0 < A \le 4(p-1)$, in which
case we shall show regularity by constructing a barrier.
Let
\begin{alignat*}{2}
   k &= 4(p-1), &\quad  f(t)&=  \lt^{-a-1}, \\
   a & = \frac{n+p-2}{k}>0, &\quad  h(t)& =2 \lt^{-a},
\end{alignat*}
where we only consider $t \in \bigl[-\tfrac{1}{3},0\bigr)$ throughout
the proof.  
We see that
\begin{equation}   \label{eq-nonpos-f'-h'}
   f'(t) = - \frac{a+1}{|t| \lt^{a+2}}  <0
\quad \text{and} \quad
   h'(t) = - \frac{2a}{|t| \lt^{a+1}}  =-\frac{2a}{|t|}f(t) <0.
\end{equation}
We want to show that 
\[
u(\xi)=-f(t)e^{|x|^2/k|t|} + h(t)
\]
is a barrier at $\xi_0$,
where we write $\xi=(x,t)$ from now on.

First, note that $u \in C^2(\overline{\Theta} \setm \{\xi_0\})$ and
it is positive if and only if 
\[
   2\lt  >  e^{|x|^2/k|t|},
\] 
i.e.\ if and only if
\[
  |x|^2 < k|t| \log {\lt} + k|t| \log 2,
\]
which holds in $\overline{\Theta} \setm \{\xi_0\}$ since $A \le k$.
Moreover,
\[
    \lim_{\Theta \ni \xi \to \xi_0} u(\xi)=0.
\]
It remains to show that $u$ is superparabolic in $\Theta$, to conclude
that $u$ is a barrier.
Since $u \in C^2(\Theta)$ we will show this
using Proposition~\ref{prop-classical}.
To this end, we see that
\begin{align*}
   \nabla u(\xi) & = - \frac{2f(t)}{k|t|} e^{|x|^2/k|t|} x, \\
   |\nabla u(\xi)|^{p-2} \nabla u(\xi) 
        & = -  \biggl(\frac{2f(t)}{k|t|}\biggr)^{p-1} e^{(p-1)|x|^2/k|t|} 
   |x|^{p-2} x, \\
   \Delta_p u(\xi) & = -  \biggl(\frac{2f(t)}{k|t|}\biggr)^{p-1} e^{(p-1)|x|^2/k|t|} 
   \biggl( \frac{2(p-1)}{k|t|} |x|^p + (n+p-2)|x|^{p-2}\biggr).
\end{align*}
Thus,  if in addition $\nabla u (\xi) \ne 0$, we have 
\[ 
   \Delta_p^N u(\xi)  = - \frac{2f(t)}{k|t|} e^{|x|^2/k|t|} 
   \biggl( \frac{2(p-1)}{k|t|} |x|^2 + n+p-2\biggr).
\]
Also 
\begin{equation}  \label{eq-u-t}
    u_t(\xi)= e^{|x|^2/k|t|} 
    \biggl( 
    -f'(t) - f(t) \frac{|x|^2}{kt^2} + h'(t) e^{-|x|^2/k|t|} 
    \biggr), 
\end{equation}
which together yields,
still provided that  $\nabla u (\xi) \ne 0$,
\begin{align*}
  u_t(\xi) - \Delta_p^N u(\xi) 
  &= e^{|x|^2/k|t|} \biggl( -f'(t)
    - \frac{f(t)|x|^2}{k^2t^2} (k- 4(p-1))  \\
    & \quad + \frac{2f(t)(n+p-2)}{k|t|}
      + h'(t) e^{-|x|^2/k|t|} \biggr). 
\end{align*}      
Using \eqref{eq-nonpos-f'-h'} and that $k=4(p-1)$,
we then obtain that 
\begin{align*}
u_t(\xi) - \Delta_p^N u(\xi) 
 & \ge   e^{|x|^2/k|t|}  \biggl(
    \frac{2f(t)(n+p-2)}{k|t|}
      + h'(t) 
    \biggr) \\
    & = e^{|x|^2/k|t|} f(t) 
    \biggl( \frac{2(n+p-2)}{k|t|}
      - \frac{2a}{|t|} 
    \biggr)  =0.
\end{align*}

When $\nabla u(x,t)=0$, we see that $x=0$. 
Moreover,
\[
\partial_i \partial_j  u(x,t) = - \frac{2f(t)}{k|t|} e^{|x|^2/k|t|}
    \biggl(\de_{ij} + \frac{2x_ix_j}{k|t|} \biggr),
\]
so $D^2 u(0,t)$
is negative definite, as $f(t)$ and $k$ are positive.
The requirements in Proposition~\ref{prop-classical} are thus met
  (even though $u_t(0,t)<0$ for $t$ close to $0$),
so $u$ is superparabolic in $\Theta$.
Hence it is a barrier, and Theorem~\ref{thm-barrier-and-more}
shows that $\xi_0$ is regular if $A \le 4(p-1)$.

\medskip

Now we turn to the case $A > 4 (p-1)$, in which case we shall show irregularity
by producing a so-called ``irregularity barrier''. 
To be more precise, $u$ is  an ``irregularity barrier'' 
if it can be used as a comparison function to show that 
the upper Perron solution does not attain its boundary values 
continuously at $\xi_0=(0,0)$. 
We will construct $u$ such that 
\begin{enumerate}
\renewcommand{\theenumi}{\textup{(\roman{enumi})}}
\item $u$ is subparabolic in 
$\Theta'=U \cap \Theta$  for some open neighbourhood $U$
  of $\xi_0$;
\item $u$ has an extension to $\overline{\Theta}'$ so 
that both $u|_{\overline{\Theta}' \setm \{\xi_0\}}$ and $u|_{\bdy \Theta'}$
are continuous;
\item $\lim_{t \to 0-}  u(0,t)=0 > u(\xi_0)$.
\end{enumerate}

First choose $k$ such that
\[
      4(p-1) < k< A
\]
and let 
\[
     a=\frac{A}{k}-1 >0
     \quad  \text{and} \quad
     b=\frac{4(n+p-2)}{k}>0.
\]
Note that the parameters $a$ and $k$, as well as the function $h$ below, are
  not the same as in the first part of the proof.
The functions
\[
f(t)=  \lt^{-a-1}
\quad \text{and} \quad
h(t)=\frac{2b}{a\lt^{a/2}},
\]
are positive for $-\tfrac13 \le  t < 0$,
considered in this proof.
Moreover,  
\[
   f'(t) = - \frac{f(t)}{|t|} \frac{a+1}{\lt} 
\quad \text{and} \quad
h'(t) = - \frac{b}{|t|\lt^{1+a/2}} 
= - \frac{bf(t)}{|t|} \lt^{a/2}.
\]
We want to show that 
\[
u(\xi)=-f(t)e^{|x|^2/k|t|} + h(t) 
\]
is an ``irregularity barrier'' for small enough $t$.
We first observe that 
\[
    \lim_{t \to 0-}  u(0,t)=0,
\]
while if $\xi \in \bdy \Theta$ and $-\tfrac13 < t<0$, then 
since $-a-1+A/k=0$,
\begin{align*}
    u(\xi) &=-\lt^{-a-1} e^{(A/k)\log \ltsub} + h(t)\\
       &= -\lt^{-a-1+A/k} + h(t) =h(t)-1  \to -1, \quad \text{as }t \to 0.
\end{align*}
We will show that  
$u$ is subparabolic in 
\[
\Theta':=\{(x,t) \in \Theta : t > -\tau\},
\]
for some  $0 < \tau< \tfrac{1}{3}$ which will be determined later.
For now, we 
take this fact for granted and show how it  implies that 
$u$ is an ``irregularity barrier'' in $\Theta'$,
and how this yields the irregularity of $\xi_0$.
Let 
\[
 \ut(\xi) = \begin{cases}
    u(\xi), & \text{if } \xi \in \bdy \Theta' \setm \{\xi_0\}, \\
    -1, & \text{if } \xi=\xi_0.
     \end{cases}
\]
Observe that $\ut \in C (\bdy \Theta')$, and
let $v \in \UU_{\ut}$.
Then, 
\[
   \limsup_{\Theta' \ni \zeta  \to \xi} u(\zeta)=u(\xi) 
   \le \liminf_{\Theta' \ni \zeta  \to \xi} v(\zeta)
   \quad \text{for all } \xi \in \bdy \Theta' \setm \{\xi_0\}.
\]
Hence, 
by Theorem~\ref{thm-parabolic-comp}, $u \le v$ in $\Theta'$,
and thus $u \le \uP_{\Theta'} \ut$  in $\Theta'$.
Therefore
\[
    \limsup_{\Theta' \ni \xi \to \xi_0} \uP_{\Theta'} \ut(\xi) 
    \ge      \lim_{t \to 0-}  u(0,t)
    =0 > \ut(\xi_0),
\]
showing that 
$\xi_0$ is irregular with respect to $\Theta'$.
By Theorem~\ref{thm-barrier-and-more}\,\ref{k-local}, 
$\xi_0$ is irregular also with respect to $\Theta$.

It remains to verify that $u$ is subparabolic in $\Theta'$, 
if $\tau$ is small enough.
As in the first part, we get, provided that $\nabla u(\xi) \ne 0$, 
\begin{align*}
  u_t(\xi) - \Delta_p^N u(\xi) 
  &= e^{|x|^2/k|t|}  
    \biggl(
    -f'(t) -  \frac{f(t)|x|^2}{k^2t^2} (k- 4(p-1))  \\
    & \quad + \frac{2f(t)(n+p-2)}{k|t|}
      + h'(t) e^{-|x|^2/k|t|} 
    \biggr) \\
  & = \frac{e^{|x|^2/k|t|} f(t)}{|t|}
    \biggl(
    \frac{a+1}{\lt} -  \frac{c|x|^2}{k^2|t|}  
     + \frac{b}{2} 
       - b e^{-|x|^2/k|t|} \lt^{a/2}
    \biggr),
\end{align*}
where $c=k- 4(p-1)>0$.
We will need three 
conditions on $\tau$.
The first is that it is so small that
\[
    \frac{ a+1}{\lt} \le  \frac{b}{2}
    \quad \text{for } -\tau < t <0,
\]
which we  assume from now on.
To show that $u_t - \Delta_p^N u\le 0$, it therefore suffices to verify that
  \begin{equation}   \label{eq-verify-bracket}
   \frac{c|x|^2}{k^2|t|} + b e^{-|x|^2/k|t|} \lt^{a/2} \ge b.
  \end{equation}
If $|x|^2 \le \tfrac12 ak |t| \log \lt$, then  
\[
b e^{-|x|^2/k|t|} \lt^{a/2} \ge b e^{-\frac12 a \log \ltsub} \lt^{a/2} = b,
\]
while if $|x|^2 > \tfrac12 ak |t| \log \lt$, then 
\[
\frac{c|x|^2}{k^2|t|} \ge \frac{ac}{2k} \log \lt \ge b
    \quad \text{for } -\tau < t <0,
\] 
provided that $\tau$ is small enough.
Hence, \eqref{eq-verify-bracket} holds in both cases.

Moreover, if $\nabla u(x,t)=0$, then $x=0$, and  from~\eqref{eq-u-t}
we see that
\[
u_t(0,t) = h'(t)-f'(t)
  =\frac{f(t)}{|t|} \biggl( \frac{a+1}{\lt} - b\lt^{a/2} \biggr) <0
    \quad \text{for } -\tau < t <0,
  \]
provided that $\tau$ is small enough.

Hence $u$ is subparabolic in $\Theta'$, by Proposition~\ref{prop-classical},
if $\tau$
is chosen small enough, which concludes the proof.
\end{proof}

\end{document}